\documentclass[11pt]{amsart}

\usepackage{amsfonts}
\usepackage{amsmath}
\usepackage{amssymb}
\usepackage{amscd}
\usepackage{mathrsfs}  
\usepackage{amsbsy}
\usepackage{amsthm}
\usepackage{graphicx}
\usepackage{fancyhdr}
\usepackage{epsfig}
\usepackage{color}
\usepackage{xy}
\usepackage{enumitem}
\usepackage{bbm}

\usepackage[OT2,OT1]{fontenc}  
\newcommand\cyr{%
\renewcommand\rmdefault{wncyr}%
\renewcommand\cFdefault{wncyss}%
\renewcommand\encodingdefault{OT2}%
\normalfont \selectfont} \DeclareTextFontCommand{\textcyr}{\cyr}

\newcommand{\be}{\begin{equation}}
\newcommand{\ee}{\end{equation}}

\newcommand{\bes}{\begin{equation*}}
\newcommand{\ees}{\end{equation*}}

\newcommand{\bH}{\mathbb{H}}

\newcommand{\N}{\mathbb{N}}

\newcommand{\R}{\mathbb{R}}

\newcommand{\A}{\mathbb{A}}

\newcommand{\cF}{\mathcal{F}}

\newcommand{\cL}{\mathcal{L}}

\newcommand{\Span}{\mathrm{span\,}}

\newcommand{\ml}{\vskip 5pt\noindent}

\newcommand{\wt}[1]{{\widetilde{#1}}}

\newcommand{\bx}{\boldsymbol{x}}
\newcommand{\mbS}{\mathbb{S}}

\renewcommand{\rmdefault}{cmr} 
\newtheorem{theorem}{Theorem}[section]
\theoremstyle{plain}

\newtheorem*{notation*}{Notation}
\numberwithin{equation}{section}

\DeclareMathOperator\Lip{Lip}

\DeclareMathOperator\Sc{Sc}
\DeclareMathOperator\Ve{Vec}

\DeclareMathOperator{\PV}{\mathrm{PV}}

\newcommand{\cFn}{\cF(X, \R_n)}

\newcommand{\n}[1]{{\left\|{#1}\right\|}}
\newcommand{\abs}[1]{\left\vert{#1}\right\vert}

\newcommand{\vertiii}[1]{{\left\vert\kern-0.25ex\left\vert\kern-0.25ex\left\vert #1 
    \right\vert\kern-0.25ex\right\vert\kern-0.25ex\right\vert}}

\begin{document}

\title[Clifford-Valued Fractal Interpolation]{Clifford-Valued Fractal Interpolation}

\author{Peter R. Massopust}
\address{Center of Mathematics, Technical University of Munich, Boltzmannstr. 3, 85748 Garching b. Munich, Germany}
\email{massopust@ma.tum.de}

\begin{abstract}
In this short note, we merge the areas of hypercomplex algebras with that of fractal interpolation and approximation. The outcome is a new holistic methodology that allows the modelling of phenomena exhibiting a complex self-referential geometry and which require for their description an underlying algebraic structure.
\end{abstract}

\maketitle 

\section{Introduction}
In this short note, we merge two areas of mathematics: the theory of hypercomplex algebras as exemplified by Clifford algebras and the theory of fractal approximation or interpolation.

In recent years, hypercomplex methodologies have found their way into many applications one of which is digital signal processing. See, for instance \cite{A,SW} and the references given therein. The main idea is to use the multidimensionality of hypercomplex algebras to model signals with multiple channels or images with multiple color values and to use the underlying algebraic structure of such algebras to operate on these signals or images. The results of these algebraic or analytic operations produce again elements of the hypercomplex algebra. This holistic approach cannot be performed in finite dimensional vector spaces as these do not possess an intrinsic algebraic structure. 

On the other hand, the concept of fractal interpolation has been employed successfully in numerous applied situations over the last decades. The main purpose of fractal interpolation or approximation is to take into account complex geometric self-referential structures and to employ approximants that are wellsuited to model these types of structures. These approximants or interpolants are elements of vector spaces and cannot be operated on in an algebraic way to produce the same type of object. Hence, the need for an extension of fractal interpolation to the hypercomplex setting. An initial investigation into the novel concept of hypercomplex iterated function system was already undertaken in \cite{m5} albeit along a different direction.

The structure of this paper is as follows. In Section 2, we give a brief introduction of Clifford algebras and mention a few items from Clifford analysis. In the third section, we review some techniques and state relevant results form the theory of fractal interpolation in Banach spaces. These techniques are then employed in Section 4 to a Clifford algebraic setting. The next section briefly mentions a special case of Clifford-valued fractal interpolation, namely that based on paravector-valued functions. In the last section, we provide a brief summary and mention future research directions.
\section{A Brief Introduction to Clifford Algebra and Analysis}\label{sec2}
In this section, we provide a terse introduction to the concept of Clifford algebra and analysis and introduce only those items that are relevant for the purposes of this paper. For more details about Clifford algebra and analysis, the interested reader is referred to, for instance, \cite{BRS,Clifford,Clifford2,GHS,Hyper,Krav} and to, i.e., \cite{CSS,CSS2,GHS1,HQW} for its ramifications. 

To this end, denote by $\{e_1, \ldots, e_n\}$ the canonical basis of the Euclidean vector space $\R^n$. The real Clifford algebra, $\R_n$, generated by $\R^n$ is defined by the multiplication rules 
\be\label{eq2.1}
e_i e_j + e_j e_i = -2 \delta_{ij}, \quad i,j\in \{1,\ldots, n\} =: \N_n, 
\ee
where $\delta_{ij}$ is the Kronecker symbol. 

An element $x\in \R_n$ can be represented in the form $x = \sum\limits_{A} x_A e_A$ with $x_A\in \R$ and $\{e_A : A\subseteq \N_n\}$, where $e_A := e_{i_1} e_{i_2} \cdots e_{i_m}$, $1\leq i_1 < \cdots < i_m \leq n$, and $e_\emptyset =: e_0 := 1$. Thus, the dimension of $\R_n$ regarded as a real vector space is $2^n$. The rules defined in \eqref{eq2.1} make $\R_n$ into a in general noncommutative algebra, i.e., a real vector space together with a bilinear operation $\R_n \times \R_n\to\R_n$.

A conjugation on Clifford numbers is defined by $\overline{x} := \sum\limits_{A} x_A \overline{e}_A$ where $\overline{e}_A := \overline{e}_{i_m} \cdots \overline{e}_{i_1}$ with $\overline{e}_i := -e_i$ for $i\in\N_n$, and $\overline{e}_0 := e_0 = 1$. In this context, one also has
\be\label{eq2}
e_0 e_0 = e_0 = 1\quad\text{and}\quad e_0 e_i = e_i e_0 = e_i. 
\ee
The Clifford norm of the Clifford number $x = \sum\limits_{A} x_A e_A$ is defined by 
\[
|x| := \left(\sum\limits_{A\subseteq\N_n} |x_A|^2\right)^{1/2}. 
\]

In the following, we consider Clifford-valued functions $f:G \subseteq\R^m\to\R_n$, where $G$ is a nonempty open domain. For this purpose, let $X$ be $G$ or any suitable subset of $\overline{G}$. Denote by $\cF(X)$ any of the following functions spaces: $C^k (X), C^{k,\alpha}(X), L^p(X), W^{s,p}(X), B^s_{p,q}(X), F^s_{p,q}(X)$ where
\begin{enumerate}
\item $C^k (X)$, $k\in\N_0 := \{0\}\cup\N$, is the Banach space of $k$-times continuously differentiable $\R$-valued functions;
\item $C^{k,\alpha}(X)$, $k\in\N_0$, $0 < \alpha \leq 1$, is the Banach space of $k$-times continuously differentiable $\R$-valued functions whose $k$-th derivative is Hölder continuous with Hölder exponent $\alpha$;
\item $L^p(X)$, $1\leq p < \infty$, are the Lebesgue spaces on $X$; 
\item $W^{s,p}(X)$, $s\in\N$ or $s>0$, $1\leq p < \infty$, are the Sobolev-Slobodeckij spaces.
\item $B^s_{p,q}(X)$, $1\leq p,q < \infty$, $s>0$, are the Besov spaces;
\item $F^s_{p,q}(X)$, $1\leq p,q < \infty$, $s>0$, are the Triebel-Lizorkin spaces.
\end{enumerate}

The  real vector space $\cF(X, \R_n)$ of $\R_n$-valued functions over $X$ is defined by 
\[
\cF(X, \R_n) := \cF(X) \otimes_\R \R_n. 
\]
This linear space becomes a Banach space when endowed with the norm
\[
\n{f} := \left(\sum_{A\subseteq\N_n} \n{f_A}^2_{\cF(X)}\right)^{1/2}.
\]

It is known \cite[Remark 2.2. and Proposition 2.3.]{GS} that $f\in \cF(X,\R_n)$ iff 
\be\label{eq2.3}
f = \sum_{A\subseteq\N_n} f_A e_A
\ee
with $f_A\in \cF(X)$. Furthermore, functions in $\cFn$ inherit all the topological properties such as continuity and differentiability from the functions $f_A\in\cF(X)$.
\section{Some Results From Fractal Interpolation Theory}\label{sec3}
In this section, we briefly summarize fractal interpolation and the 
Read-Bajrakterevi\'{c} operator. For a more detailed introduction to fractal geometry and its subarea of fractal interpolation, the interested reader is referred to the following, albeit incomplete, list of references: \cite{B1,B2,BHVV,bhm,bhm2,bedford,DLM,dubuc1,dubuc2,F,H,LDV,massopust1,SB,SB1}.

To this end, let $X$ be a nonempty bounded subset of the Banach space $\R^m$. Suppose we are given a finite family $\{L_i\}_{i = 1}^{N}$ of injective nontrivial contractions $X\to X$ generating a partition of $X$ in the sense that
\begin{align}
&\forall\;i, j\in \N_N, i\neq j: L_i(X)\cap L_j(X) = \emptyset;\label{c1}\\
&X = \bigcup_{i=1}^N L_i(X).\label{c2}
\end{align}
For simplicity, we write $X_i := L_i(X)$. Here and in the following, we always assume that $1<N\in\N$.

The purpose of fractal interpolation is to obtain a unique global function 
\[
\psi: X = \bigcup\limits_{i=1}^N X_i\to\R 
\]
belonging to some prescribed Banach space of functions $\cF(X)$ and satisfying $N$ functional equations of the form
\be\label{psieq}
\psi (L_i (x)) = q_i (x) + s_i (x) \psi (x), \quad x\in X,\ i\in \N_N,
\ee
where for each $i\in\N_N$, $q_i\in\cF(X)$ and $s_i:X\to\R$ are given functions. In addition, we require that $s_i$ is bounded and satisfies $s_i \cdot f\in\cF(X)$ for any $f\in\cF(X)$, i.e., $s_i$ is a multiplier for $\cF(X)$. It is worthwhile mentioning that Eqn. \eqref{psieq} reflects the self-referential or fractal nature of the global function $\psi$.

The idea behind obtaining $\psi$ is to consider \eqref{psieq} as a fixed point equation for an associated affine operator acting on $\cF(X)$ and to show that the fixed point - should it exist - is unique. (Cf. also \cite{SB}.)

For this purpose, define an affine operator $T: \cF(X)\to \cF(X)$, called a Read-Bajractarevi\'c (RB) operator, by 
\be\label{eq3.4}
T f  := q_i\circ L_i^{-1} + s_i\circ L_i^{-1}\cdot f\circ L_i^{-1}, 
\ee 
on $X_i$, $i\in \N_N$, or, equivalently, by
\begin{align*}
T f &= \sum_{i=1}^N q_i\circ L_i^{-1}\, \chi_{X_i} + \sum_{i=1}^N s_i\circ L_i^{-1}\cdot f\circ L_i^{-1}\, \chi_{X_i}\\
&= T(0) + \sum_{i=1}^N s_i\circ L_i^{-1}\cdot f\circ L_i^{-1}\, \chi_{X_i},\quad x\in X,
\end{align*}
where $\chi_S$ denotes the characteristic or indicator function of a set $S$: $\chi_S(x) = 1$, if $x\in S$, and $\chi_S(x) = 0$, otherwise. 

Then, \eqref{psieq} is equivalent to showing the existence of a unique fixed point $\psi$ of $T$: $T\psi = \psi$. The existence of a unique fixed point follows from the Banach Fixed Point Theorem once it has been shown that $T$ is a contraction on $\cF(X)$. 

The RB operator $T$ is a contraction on $\cF(X)$ if there exists a constant $\gamma_{\cF(X)}\in [0,1)$ such that for all $f,g\in\cF(X)$
\begin{align*}
\n{Tf - Tg}_{\cF(X)} &= \n{\sum_{i=1}^N s_i\circ L_i^{-1}\cdot (f-g)\circ L_i^{-1}\, \chi_{X_i}}_{\cF(X)}\\
& \leq \gamma_{\cF(X)} \n{f - g}_{\cF(X)}
\end{align*}
holds. Here, $\n{\cdot}_{\cF(X)}$ denotes the norm on $\cF(X)$.

Should such a unique fixed point $\psi$ exist then is termed a \emph{fractal function of type $\cF(X)$} as its graph is in general a fractal set. 

Now, let $\cF(X)$, for an appropriate $X$, denote one of the following Banach space of functions: the Lebesgue spaces $L^p(X)$, the smoothness spaces $C^k(X)$, Hölder spaces $C^{k,\alpha}(X)$, Sobolev-Slobodeckij spaces $W^{s,p}(X)$, Besov spaces $B^s_{p,q}(X)$, and Triebel-Lizorkin spaces $\cF^s_{p,q}(X)$.

The following results were established in a series of papers \cite{PRM,massopust, massopust1,m2,m3}.

\begin{theorem}
Let $L_i$, $i\in\N_N$, be defined as in \eqref{c1} and \eqref{c2}. Further let $q_i\in\cF(X)$ and let $s_i:X\to\R$ be bounded and a pointwise multiplier for $\cF(X)$. Define $T:\cF(X)\to\cF(X)$ as in \eqref{eq3.4}. Then there exists a constant $\gamma_\cF\in [0,1)$ depending on $m$, the indices defining $\cF(X)$, $\Lip(L_i)$, and $\n{s_i}_{L^\infty}$ such that $\n{T f} \leq \gamma_{\cF}\; \n{f}$, for all $f\in\cF(X)$. Hence, $T$ has a unique fixed point $\psi\in\cF(X)$ which is referred to as a fractal function of class $\cF(X)$. 
\end{theorem}
\section{Clifford-Valued Fractal Interpolation}\label{sec4}
In this section, we introduce the novel concept of Clifford-valued fractal interpolation. To this end, we refer back to Section \ref{sec2} and the definition of $X$ and $\cF(X)$. 

We consider here only the case $m=1$ and leave the extension to higher dimensions to the diligent reader. According to which function space $\cF(X)$ represents, $X$ is either an open, half-open or closed interval of finite length.

Assume that there exist $N$, $1<N\in\N$, nontrivial contractive injections $L_i:X\to X$ such that $\{L_1( X), \ldots, L_N( X)\}$ forms a partition of $ X$, i.e, that
\begin{enumerate}
\item[(P1)]\label{P1} $L_i( X) \cap L_j( X) = \emptyset$, for $i\neq j$;
\item[(P2)]\label{P2} $ X = \bigcup\limits_{i\in\N_N} L_i( X)$.
\end{enumerate}
As above, we write $X_i := L_i(X)$, $i\in\N_N$.

On the spaces $\cFn$, we define an RB operator $\wt{T}$ as follows. Let $f\in\cFn$ with $f = \sum\limits_{A\subseteq\N_n} f_A e_A$, where $f_A\in\cF(X)$, for all $A\subseteq\N_n$. Let $T:\cF(X)\to\cF(X)$ be an RB operator of the form \eqref{eq3.4}. Then,
\be\label{eq4.1}
\wt{T}f := \sum\limits_{A\subseteq\N_n} T(f_A) e_A\in\cFn,
\ee
provided that $T(f_A)\in\cF(X)$ for all $A\subseteq\N_n$. Under the latter assumption and the supposition that $T$ is contractive on $\cF(X)$ with Lipschitz constant $\gamma_{\cF(X)}$, we obtain for $f,g\in\cFn$
\begin{align*}
\n{\wt{T}f - \wt{T}g}^2 &= \sum\limits_{A\subseteq\N_n} \n{Tf_A - Tg_A}_{\cF(X)}^2\\
&\leq \gamma_{\cF(X)}^2 \sum\limits_{A\subseteq\N_n} \n{f_A - g_A}^2_{\cF(X)}\\
&= \gamma_{\cF(X)}^2 \n{f-g}^2.
\end{align*}
Hence, $\wt{T}$ is also contractive on $\cFn$ and with the same Lipschitz constant $\gamma_{\cF(X)}$.

The following diagram illustrates the above approach. 
\be\label{diagram}
\begin{CD}
\cF(X) @>T>> \cF(X)\\
@VV\otimes_\R \R_nV                  @VV\otimes_\R \R_nV\\
\cFn @>\wt{T}>>  \cFn
\end{CD}
\ee

The next theorem summarizes the main result.
\begin{theorem}\label{thm4.1}
Let $X\subset\R$ be as mentioned above. Further, let nontrivial injective contractions $L_i:X\to X$, $i\in\N_N$, be given such that (P1) and (P2) are satisfied. Let $\cF(X)$ be any one of the function spaces defined in Section \ref{sec2}. 

On the space $\cFn = \cF(X)\otimes_\R \R_n$ define an RB operator $\wt{T}:\cFn$ $\to\cFn$ by
\[
\wt{T} f := \sum\limits_{A\subseteq\N_n} T(f_A) e_A,
\]
where $T:\cF(X)\to\cF(X)$ be an RB operator of the form \eqref{eq3.4}.

If $T:\cF(X)\to\cF(X)$ is a contractive RB operator on $\cF(X)$ with Lipschitz constant $\gamma_{\cF(X)}$, then $\wt{T}$ is also contractive on $\cFn$ with the same Lipschitz constant. 

Furthermore, the unique fixed point $\psi\in\cFn$ satisfies the Clifford-valued self-referential equation 
\[
\psi (L_i (x)) = q_i (x) + s_i (x) \psi (x), \quad x\in X,\ i\in \N_N.
\]
\end{theorem}

\begin{proof}
The validity of these statements follows directly from the above elaborations.
\end{proof}

For the sake of completeness, we now list the Lipschitz constants $\gamma_{\cF(X)}$ for the functions spaces listed in Section \ref{sec2} in the case $m=1$. The conditions are $\gamma_{\cF(X)} < 1$. Note that the expressions are different for the case $m>1$.

\begin{enumerate}
\item $C^k (X)$: $\gamma_{C^k(X)} = \max\{\Lip (L_i)^{-(k+1)} \n{s_i}_{L^\infty} : i\in\N_N\}$.
\ml
\item $C^{k,\alpha}(X)$: $\gamma_{C^{k,\alpha} (X)} = \max\{\Lip (L_i)^{-(k+\alpha)} \n{s_i}_{L^\infty} : i\in\N_N\}$.
\ml
\item $L^p(X)$: $\gamma_{L^p(X)} = \displaystyle{\sum\limits_{i\in\N_N}} \Lip (L_i) \n{s_i}^p_{L^\infty}$.
\ml
\item $W^{s,p}(X)$: $\gamma_{W^{s,p}(X)} = \displaystyle{\sum\limits_{i\in\N_N}} \Lip (L_i)^{1 - s p} \n{s_i}^p_{L^\infty}$.
\ml
\item $B^s_{p,q}(X)$: $\gamma_{B^s_{p,q}(X)} = \displaystyle{\sum\limits_{i\in\N_N}} \Lip (L_i)^{(1/p - s) q} \n{s_i}^q_{L^\infty}$.
\ml
\item $F^s_{p,q}(X)$: $\gamma_{F^s_{p,q}(X)} = \displaystyle{\sum\limits_{i\in\N_N}} \Lip (L_i)^{1 - s p} \n{s_i}^p_{L^\infty}$.
\ml
\end{enumerate}

The geometric interpretation of $\wt{T}$ lies at hand: Each of the functions $f_A$ is contracted by $T$ along the direction in $\R_n$ determined by $e_A$. There is no mixing taking place between different directions. This provides a holistic representation of features necessitating such a structure as, for instance, multichannel data or multicolored images.

\section{Paravector-Valued Functions}\label{sec5}

An important subspace of $\R_n$ is the space of paravectors. These are Clifford numbers of the form $x = x_0 + \sum\limits_{i=1}^n x_i e_i$. The subspace of paravectors is denoted by $\A_{n+1} := \Span_\R\{e_0, e_1, \ldots, e_n\} = \R\oplus \R^n$. Given a Clifford number $x\in \R_n$, we assign to $x$ its paravector part by means of the mapping $\pi: \R_n\to \A_{n+1}$, $x \mapsto x_0 + \sum\limits_{i=1}^n x_i e_i$. 

Note that each paravector $x$ can be identified with an element $(x_0, x_1, \ldots, $ $x_n) =: (x_0, \bx)\in \R\times \R^n$. For many applications in Clifford theory, one therefore identifies $\A_{n+1}$ with $\R^{n+1}$. Although as point sets, these two sets are identical but differ considerably in their algebraic structures. For instance, every $x\in \A_{n+1}$ has an inverse whereas there is no such object for a vector $v\in\R^{n+1}$.

We also notice that $\A_{n+1}$ is not necessarily closed under multiplication unless a multiplication table \cite{Hyper} is defined or $n=3$,  in which case $\A_4 = \bH$, the non-commutative division algebra of quaternions. $\A_{n+1}$ endowed with a multiplication table produces in general a nonassociative noncommutative algebra. See \cite{A} for a suitability investigation of such algebras in the area of digital signal processing.

The scalar part, $\Sc$, and vector part, $\Ve$, of a paravector $\A_{n+1}\ni x = x_0 + \sum\limits_{i=1}^n x_i e_i$ is given by $x_0$ and $\boldsymbol{x} = \sum\limits_{i=1}^n x_i e_i$, respectively. 

Given a Clifford number $x\in \R_n$, we assign to $x$ its paravector part, $\PV(x)$, by means of the mapping $\pi: \R_n\to \A_{n+1}$, $x \mapsto x_0 + \sum\limits_{i=1}^n x_i e_i =: \PV(x)$.

A function $f:\A_{n+1}\to\A_{n+1}$ is called a paravector-valued function. Any such function is of the form 
\be\label{eq5.1}
f(x) = f_0 (x) + \sum\limits_{i=1}^n f_i (x) e_i, 
\ee{align}
where $f_a:\R\times\R^{n}\to\R$, $a\in\{0,1,\ldots, n\}$. The expression \eqref{eq5.1}  for a paravector-valued function can also be written in the more succinct form 
\[
f (x + \bx) = f_0(x_0, |\bx|) + \omega (\bx)f_1(x_0,|\bx|),
\]
where now $f_0, f_1:\R\times\R^n\to\R$ and $\omega (\bx) := \frac{\bx}{|\bx|}\in \mbS^n$ with $\mbS^n$ denoting the unit sphere in $\R^n$. For some properties of  paravector-valued functions, see, for instance \cite{GS,sproessig}.

Prominent examples of paravector-valued functions are for instance the exponential and sine functions \cite{sproessig} for $x\in \A_{n+1}$:
\begin{align*}
\exp (x) =& \exp(x_0)\left(\cos \abs{\bx} + \omega(\bx) \sin \abs{\bx}\right),\\
\sin (x) = & \sin x_0 \cosh \abs{\bx} + \omega(\abs{x}) \sinh \abs{\bx}.
\end{align*}

A large class of paravector-valued functions is given by right-linear linear transformations. To this end, let $M_k (\A_{n+1})$ be the right module of $k\times k$-matrices over $\A_{n+1}$. Every element $H = (H_{ij})$ of $M_k (\A_{n+1})$ induces a right linear transformation $L: \A_{n+1}^k\to \R_n^k$ via $L(x) = H x$  defined by $L(x)_i = \sum\limits_{j=1}^k H_{ij} x_j$, $H_{ij}\in \A_{n+1}$. To obtain an endomorphism $\cL:\A_{n+1}^k\to\A_{n+1}^k$, we set $\cL(x)_i := \pi(L(x)_i)$, $i=1, \ldots, k$. In this case, we write $\cL = \pi\circ L$. For example, if $n:=3$ (the case of real quaternions) $L: \A_{4}^k\to \A_{4}^k$ and thus $\cL = L$.

Theorem \ref{thm4.1} applies also to paravector-valued functions and thus provides a framework for paravector-valued fractal interpolation as well and relevant associated function spaces for appropriate $X$ are defined in an analogous fashion as above. 

To this end, let $\cF(X)$ be for instance one of the function spaces listed in Section \ref{sec2}. Then, 
\[
\cF(X, \A_{n+1}) := \cF(X)\otimes_\R \A_{n+1}
\]
and an element $f$ of $\cF(X, \A_{n+1})$ has therefore the form
\[
f = \sum_{k=0}^n f_k e_k.
\]
Theorem \ref{thm4.1} then asserts the existence of a paravector-valued function $\psi\in \cF(X,\A_{n+1}$ of self-referential nature:
\[
\psi (L_i (x)) = q_i (x) + s_i (x) \psi (x), \quad x\in X,\ i\in \N_N,
\]
where the functions $q_i$ and $s_i$ have the same meaning as in Section \ref{sec4}.
\section{Brief Summary and Further Research Directions}
In this short note, we have initiated the investigation of fractal interpolation into a hypercomplex setting. The main idea was to define fractal interpolants along the different directions defined by a Clifford algebra $\R_n$ and use the underlying algebraic structure to manipulate the hypercomplex fractal object to yield another hypercomplex fractal object. 

There are several extensions of this first initial approach:
\begin{enumerate}
\item Define -- under suitable conditions -- RB operators acting directly on appropriately defined function spaces $\cF(X, \R_n)$ instead of resorting to the ``component'' RB operators.
\item Provide a \emph{local} version of the defined hypercomplex fractal interpolation in the sense first defined in \cite{BH} and further investigated in, i.e., \cite{bhm,m2}.
\item Construct \emph{nonstationary} approaches to Clifford-valued fractal interpolation in the spirit of \cite{m4}.
\item Extend the notion of hypercomplex fractal interpolation to systems of function systems as described in \cite{DLM,LDV}.
\end{enumerate}
\end{document}